\documentclass[a4paper, 12pt]{article} \usepackage[utf8]{inputenc} \usepackage[T2A]{fontenc} \usepackage{dsfont} \usepackage{lipsum} \usepackage{amsmath,amssymb,amsthm} \usepackage[a4paper,hmargin=2.5cm,vmargin=2.5cm]{geometry} \usepackage[english]{babel} \author{Petr Kucheriaviy} \title{On numbers not representable as $n + w(n)$}  \date{%
    National Research University Higher School of Economics, Moscow, Russia\\[2ex]%
    \today
}

\newcommand\blfootnote[1]{%
  \begingroup
  \renewcommand\thefootnote{}\footnote{#1}%
  \addtocounter{footnote}{-1}%
  \endgroup
}

\begin{document}

\newtheorem{lemma}{Lemma}
\newtheorem{theorem}{Theorem}
\newtheorem{hypo}{Hypothesis}
\newtheorem{prop}{Proposition}

\newtheorem*{hypo*}{Conjecture}

\maketitle

\blfootnote{The author was supported by the Basic Research Program of the National Research University Higher School of Economics.}

\begin{abstract}
    Let $w(n)$ be an additive non-negative integer-valued arithmetic function which is equal to $1$ on primes. We study the distribution of $n + w(n)$ $\pmod p$ and give a lower bound for the density of the set of numbers which are not representable as $n + w(n)$.
\end{abstract}


\section{Introduction}

Let $w(n)$ be an additive non-negative integer-valued arithmetic function which is equal to $1$ on primes. For example one can set $w(n)$ to be $\omega(n)$ -- the number of distinct prime factors of $n$ or $\Omega(n)$ -- the number of prime factors with multiplicity. By $E$ we denote the set of natural numbers not representable as $n + w(n)$. For the complement of $E$ we use the notation $E^c$. In this article we are concerned with the lower bound for
\[
\Xi(N) = |[1, N] \cap E|.
\]

\begin{hypo*}
\[
\Xi(N) \gg N.
\]
\end{hypo*}

We show that

\begin{theorem}
\label{main}
\[
\Xi(N) \gg \frac{N}{\log \log N}.
\]
\end{theorem}

The main idea is to consider $n + w(n)$ modulo a prime number. For a prime number $p$ and $r \in \mathbb{Z}/p \mathbb{Z}$ let us denote

\[
a(r) := \sum_{\substack{n \le N \\ n + w(n) \equiv r \pmod p}} 1.
\]

\begin{lemma}
\label{idea}
Let $R$ be a subset of $\mathbb{Z}/ p \mathbb{Z}$. Then
\[
\Xi(N) \ge \sum_{r \in R} (N/p - a(r) - 1).
\]
\end{lemma}
\begin{proof}
Let
\[
N_r := \{n \, | \, n \le N, \, n \equiv r \pmod p\}.
\]
Note that
\[
[1, N] \cap \mathbb{Z} = \coprod_{r \in \mathbb{Z}/ p \mathbb{Z}} N_r, \quad |N_r| \ge N/p - 1, \quad |N_r \cap E^c| \le a(r).
\]
Hence
\[
|N_r \cap E| = |N_r| - |N_r \cap E^c| \ge (N/p - a(r) - 1).
\]
Finally,
\[
\Xi(N) = \sum_{r \in \mathbb{Z}/ p \mathbb{Z}} |N_r \cap E| \ge \sum_{r\in R} |N_r \cap E| \ge \sum_{r \in R} (N/p - a(r) - 1).
\]
\end{proof}

In $\cite{Changa}$ M. E. Changa  studied the distribution of $w(n)$ in a residue ring.
Using the same method we study the asymptotic behavior of $a(r)$. 

It turns out that for any fixed $p$
\[
a(r) \sim N/p.
\]

That is the reason we can't prove the Conjecture directly using the observation of Lemma \ref{idea}. Nevertheless, we are able to find the second term in the asymptotic expansion of $a(r)$. Then Theorem \ref{main} follows by an appropriate choice of $p$ and $R$.

Let
\[
G(N) := |\{(n, m) \, : \, n + w(n) = m + w(m), n \le N, m \le N, n \ne m\}|.
\]

\begin{prop} \label{G}
If $w(n) \ll n^{\varepsilon}$ and $0 < \varepsilon < 1$, then
\[
G(N) \gg \Xi(N).
\]
\end{prop}

\begin{proof}
We have $w(n) \le C n^{\varepsilon}$ for some constant $C$.

Let us denote 
\[
g(n) := |\{ m \, : \, m + w(m) = n, m \le N \}|.
\]

Note that if $n \le N - C N^\varepsilon$, then $n + w(n) \le N$. Hence
\[
\sum_{n \le N} g(n) = \sum_{\substack{n \ge 1 \\ n + w(n) \le N}} 1 \ge N - C N^\varepsilon.
\]

We have
\[
G(N) \ge \sum_{n \le N} \binom{g(n)}{2} = \sum_{n \in [1, N] \cap E^c} \binom{g(n)}{2} \ge \sum_{n \in [1, N] \cap E^c} (g(n) - 1) \ge
\]
\[
(N - C N^\varepsilon) - |[1, N] \cap E^c| = \Xi(N) - C N^{\varepsilon} \gg \Xi(N).
\]
\end{proof}

In $\cite{Erdos1}$, P. Erd\H{o}s, A. S\'ark\H{o}zy, C. Pomerance showed that if $w(n)$ is the number of distinct prime factors of $n$, then

\[
G(N) \ge N \exp(-4000 \log \log N \log \log \log N)
\]
for $N$ large enough.

So Theorem \ref{main} in view of Proposition \ref{G} improves this bound and gives

\[
G(N) \gg \frac{N}{\log \log N}.
\]

\section{Lemmas}

\begin{lemma}[Perron's formula] \label{Perron}
Let
\[
F(s) = \sum_{n = 1}^{\infty} \frac{a(n)}{n^s},
\]
where the series converges absolutely for $\operatorname{Re} s > a \ge 0$, and let
\[
\sum_{n = 1}^\infty \frac{|a_n|}{n^\sigma} \ll (\sigma - a)^{-\alpha}
\]
as $\sigma \rightarrow a + 0$ for some $\alpha > 0$.
 Then for every $b > a$, $x \ge 2$ and $T \ge 2$ we have
\[
\sum_{n \le x} a(n) = \frac{1}{2 \pi i} \int_{b - iT}^{b + iT} F(s) \frac{x^s}{s} ds + R(x),
\]
where
\[
R(x) \ll \frac{x^b}{T(b-a)^\alpha} + 2^b \left( \frac{x \log x}{T} + \log \frac{T}{b} + 1 \right) \max_{x/2 \le n \le 3x/2} |a(n)|.
\]

\end{lemma}
\begin{proof}
See, for example, $\cite[\text{Theorem}~7,\text{p.}~20]{Changa_book}$ or $\cite[\text{Theorem}~1,\text{p.}~64]{Karatsuba}$.
\end{proof}

\begin{lemma}[M. E. Changa]
\label{Chang}
Let $f(n)$ be a complex valued arithmetic function such that $|f(n)| \le 1$, let the following equality hold for $\operatorname{Re} s > 1$:
\[
\sum_{n = 1}^\infty \frac{f(n)}{n^s} = (\zeta(s))^\alpha H(s),
\]
where $\alpha \in \mathbb{C}$, $|\alpha| \le 1$, 
and  the function $H(s), s = \sigma + it$, is analytic in the domain $\sigma > 1 - \frac{c_1}{\log  T} , |t| \le T$ for every sufficiently large $T$, and let the bound $H(s) \ll \log^{c_2} T$ hold in this domain. Then
\[
\sum_{n \le x} f(n) = \frac{H(1)}{\Gamma(\alpha)} x \log^{\alpha - 1} x + O(x \log^{\operatorname{Re} \alpha - 2} x).
\]
Here constant in $O$ depends only on $c_1, c_2, C$.
\end{lemma}
\begin{proof}
See $\cite[\text{Lemma}~3.1]{Changa}$.
\end{proof}

\begin{lemma} \label{Siegel0}
Let $T \ge 3$. If $\chi$ is a complex character modulo $k$ and $s = \sigma + i t$, then $L(s, \chi)$ has no zeroes in the domain
\[
\operatorname{Re} s = \sigma \ge 1 - \frac{c_3}{\log k T}, \, \, |t| \le T .
\]
If $\chi$ is a real character modulo $k$ and $s = \sigma + i t$, then $L(s, \chi)$ has no zeroes in the domain
\[
\operatorname{Re} s = \sigma \ge 1 - \frac{c_3}{\log k T}, \, \, 0 < |t|  \le T .
\]
\end{lemma}
\begin{proof}
See $\cite[\text{Theorem}~2,\text{p.}~124]{Karatsuba}$
\end{proof}

\begin{lemma}
\label{Siegel}
Let $3 \le X$. There exists at most one $3 \le k \le X$ and at most one real primitive character  $\chi_1$ modulo $k$ for which  $L(s, \chi_1)$ has a simple real zero $\beta_1$ such that
\[
\beta_1 \ge 1 - \frac{c_4}{\log X}.
 \]
\end{lemma}
\begin{proof}
See $\cite[\text{Corollary}~2,\text{p.}~131]{Karatsuba}$.
\end{proof}

We call a prime number $p < X$ unexceptional if  $L(s, \chi)$ has no real zeroes on the interval
\[
\left[ 1 - \frac{\min(c_3, c_4)}{\log X}, +\infty \right)
 \]
for every character $\chi$ modulo $p$. Otherwise we call $p$ exceptional. Note that the notion of being exceptional depends on $X$. Lemma \ref{Siegel} implies that there exist at most one exceptional $p < X$. 

\begin{lemma}
\label{short}
For every sufficiently large $X$, the interval
\[
\left[ X - X^{4/5}, X \right]
\]
contains at least two prime numbers.
\end{lemma}
\begin{proof}
See $\cite{Cons_primes}$ or $\cite[\text{Theorem}~2,\text{p.}~98]{Karatsuba}$.
\end{proof}

\begin{lemma}[Borel–Carathéodory theorem]
\label{Bor}
Let $R > 0$, and let the function $f(s)$ be analytic in the circle $|s - s_0| \le R$. Moreover assume that $\operatorname{Re} f(s) \le M$ on the circle $|s - s_0| = R$. Then
in the disk $|s - s_0| \le r < R$ we have
\[
|f(s) - f(s_0)| \le 2(M - \operatorname{Re} f(s_0)) \frac{r}{R - r}.
\]
\end{lemma}
\begin{proof}
See, for example, $\cite[\text{Lemma}~4,\text{p.}~35]{Karatsuba}$. 
\end{proof}

\begin{lemma} \label{L=O}
Let $\chi$ be a non-principal character modulo $k$. Then the following bound holds in the domain $|\sigma - 1| \le c / \log k T$, $|t| \le T$:
\[
L(s, \chi) = O(\log k T).
\]
\end{lemma}

\begin{proof}
Let us denote $S(t) = \sum\limits_{n \le t} \chi(n)$. We have
\[
\sum_{n = N + 1}^{M} \frac{\chi(n)}{n^s} = \int_{N}^M \frac{d S(t)}{t^s} = \frac{S(M)}{M^s} - \frac{S(N)}{N^s} + s \int_{N}^M 
\frac{S(t) \, d t}{t^{s + 1}}.
\]
As $M$ approaches infinity, we find that
\[
L(s, \chi) = \sum_{n = 1}^{N} \frac{\chi(n)}{n^s} - \frac{S(N)}{N^s} + s \int_{N}^\infty  \frac{S(t) \, d t}{t^{s + 1}}.
\]
Hence
\[
L(s, \chi) - \sum_{n = 1}^{N} \frac{\chi(n)}{n^s} = O(k N^{-\sigma}) + O(k T \sigma^{-1} N^{-\sigma}).
\]
If $n \le k T$, then
\[
|n^{-s}| = n^{-\sigma} \le n^{-1} (k T)^{c / \log k T} \ll n^{-1}.
\]
If we take $N = k T$, we obtain
\[
L(s, \chi) = O(\log k T).
\]
\end{proof}

\begin{lemma} \label{log L}
Let $T \ge 3$, let $\chi$ be a non-principal character modulo $k$, and let $L(s, \chi)$ be nonzero in the domain 
\[
\sigma \ge 1 - \frac{3c}{\log k T}, \, \, |t| \le 2 T;
\]
then on the boundary of the domain
\[
\sigma \ge 1 - \frac{c}{\log k T}, \, \, |t| \le T
\]
we have $\log L(s, \chi) \ll \log \log k T$.
\end{lemma}

\begin{proof}
If $\sigma \ge 1 + c/\log k T$, then
\[
|\log L(s, \chi)| = \left| \sum_{p} \sum_{m=1}^\infty  \frac{\chi^m(p)}{m p^{m s}} \right| \le \log \zeta(\sigma) \ll \log \frac{1}{\sigma - 1} \ll \log \log k T.
\]

Using Lemma \ref{L=O}, we obtain
\[
\operatorname{Re} \log L(s, \chi) = \log |L(s, \chi)| \le C' \log \log k T.
\]
for $|\sigma - 1| \le 3c/\log{k T}$, $|t| \le 2 T$.

Thus, setting $s_0 = 1 + c/\log k T + i t$, $R = 3 c / \log k T$, $r = 2c / \log k T$ in Lemma \ref{Bor}, we find that
\[
|\log L(s, \chi) - \log L(s_0, \chi)| \le 2 (C' \log \log k T - O(\log \log k T)) \ll \log \log k T.
\]
for $|s - s_0| = r$
\end{proof}

\begin{lemma}
\label{L_main}

Let $\chi$ be a non-principal character modulo $k$ and $\log k = o(\sqrt{\log x})$. Assume that $L(s, \chi)$  has no zeroes in the domain $\sigma \ge 1 - \frac{3c}{\log k T}, \, \, |t| \le 2 T$ for $T > C$. 
Let $\alpha \in \mathbb{C}$, $|\alpha| \le 1$. 
Let $f(n)$ be a complex valued arithmetic function, such that $|f(n)| \le 1$ and  let
\[
\sum_{n = 1}^\infty \frac{f(n)}{n^s} = (L(s, \chi))^\alpha H(s),
\]
for $\operatorname{Re} s > 1$. Let $H(s), s = \sigma + it$ be analytic in the domain $\sigma > 1 - \frac{3c}{\log k T} , |t| \le T$ for each $T > C$. Moreover, let $H(s) \ll \log^{c_5} T $ in this domain.

Then
\[
\sum_{n \le x} f(n) \ll x \exp\left( -c_6 \sqrt{\log x} \right).
\]
Here the constant in Vinogradov symbol depends only on $c, c_5, C$ and does not depend on $\chi$.
\end{lemma}

\begin{proof}

Substituting  $b = 1 + 1/\log{x}$, $\log T = \sqrt{\log{x}}$ into Perron's formula (Lemma \ref{Perron}), we obtain

\[
\sum_{n \le x} f(n) = \frac{1}{2 \pi i} \int_{b - i T}^{b + i T} (L(s, \chi))^\alpha H(s) \frac{x^s}{s} ds + R(x),
\]
where
\[
R(x) \ll x \exp(-c_7 \sqrt{\log x}).
\]

Let $\beta := 1 - c/\log k T$ and consider the contour $\Gamma$ consisting of segments joining in succession the points  $\beta + iT, \beta - iT, b - iT, b + i T, \beta + iT$.

Note that $(L(s, \chi))^\alpha H(s) \frac{x^s}{s}$ is well-defined and holomorphic in the interior and on the sides of $\Gamma$.

On the sides of $\Gamma$ we have
\[
|(L(s, \chi))^\alpha| = |e^{\alpha \log L(s, \chi)}| \le e^{|\log L(s, \chi)|} \ll \log^{c_8} k T,
\]
here we used Lemma \ref{log L}.

By residue theorem, we have
\[
\int_{\Gamma} (L(s, \chi))^\alpha H(s) \frac{x^s}{s} ds = 0.
\]

Estimating the integral trivially on the sides of $\Gamma$, we obtain
\[
\int_{b - i T}^{b + i T} (L(s, \chi))^\alpha H(s) \frac{x^s}{s} ds \ll x \exp\left( -c_6 \sqrt{\log x} \right).
\]
This concludes the proof. 

\end{proof}

Let us set $e(z) := \exp(2 \pi i z)$. By $\mathbb{P}$ we denote the set of prime numbers.

\begin{lemma}
\label{F_{t, p, u}}
Let $u(n)$ be a completely multiplicative function such that $|u(n)| \le 1$ for all $n$. 
Set
\[
L(s, u) := \sum_{n \ge 1} \frac{u(n)}{n^{s}}.
\]

Then for $\operatorname{Re} s > 1$ we have
\[
\sum_{n \ge 1} \frac{e \left( \frac{w(n) t}{p} \right) u(n)}{n^s} = (L(s, u))^{e(t/p)} F_{t, p, u}(s).
\]
Here $F_{t, p, u}(s)$ is holomorphic on $\operatorname{Re} (s) > 1/2$. Moreover, for every $\epsilon > 0$, there are $0 < b_{1, \epsilon} < b_{2, \epsilon}$, such that in the half-plane $\operatorname{Re} (s) > 1/2 + \epsilon$ we have
\[
b_{1, \epsilon} < |F_{t, p, u}(s)| < b_{2, \epsilon}.
\]
And  $b_{1, \epsilon}$ and $b_{2, \epsilon}$ do not depend on $p, t$, and $u$.

In particular there exist constants $0 < b_1 < b_2$ such that
\[
b_1 \le |F_{t, p, u}(1)| \le b_2
\]
for every $p$, $t$ and $u$.

\end{lemma}

\begin{proof}
Set
\[
 F_{t, p, u}(s) = (L(s, u))^{-e(t/p)} \sum_{n \ge 1} \frac{e \left( \frac{w(n) t}{p} \right) u(n)}{n^s} .
\]
By assumption, $w(n)$ is an additive function and $w(q) = 1$ for $q \in \mathbb{P}$.
Thus for $\operatorname{Re} s > 1$, we have
\[
\sum_{n \ge 1} \frac{e \left( \frac{w(n) t}{p} \right) u(n)}{n^s} = \prod_{q \in \mathbb{P}} 
\left( 1 + \frac{e(t/p) u(q)}{q^s} + \frac{e\left( \frac{w(q^2) t}{p} \right) u(q)^2}{q^{2s}} + \ldots \right)
\]
and
\[
L(s, u) = \prod_{q \in \mathbb{P}} (1 - u(q) q^{-s})^{-1}.
\]

Therefore,
\[
\log F_{t, p, u}(s) = \sum_{q \in \mathbb{P}} e(t/p) \log(1 - u(q) q^{-s}) + \log \left( 1 + \frac{e(t/p) u(q)}{q^s} + \frac{e\left( \frac{w(q^2) t}{p} \right) u(q)^2}{q^{2s}} + \ldots \right).
\]

Let us denote
\[
h_{q, t, p, u}(s) := \sum_{j \ge 2} \frac{e\left( \frac{w(q^j) t}{p} \right) u(q)^j}{q^{j s}}.
\]

This series converges uniformly on  $\operatorname{Re}(s) > 1/2 + \epsilon$. So $h_{q, t, p, u}(s)$ is analytic in the half-plane $\operatorname{Re}(s) > 1/2$. In addition, $h_{q, t, p, u}(s) = O(q^{-2 \sigma})$.

For $\operatorname{Re} s > 1$ we have
\[
\log F_{t, p, u}(s) = \sum_{q \in \mathbb{P}} e(t/p)(-u(q) q^{-s} - u(q)^2 q^{-2s}/2 - \ldots)  + \]
\[ 
(e(t/p) u(q) q^{-s} + h_{q, t, p, u}(s)) + 
\sum_{j \ge 2} \frac{(-1)^{j + 1}}{j} (e(t/p) u(q) q^{-s} + h_{q, t, p, u}(s))^{j}.
\]

If $\operatorname{Re} s \ge 1/2 + \epsilon$, we obtain
\[
\log F_{t, p, u}(s) \ll \sum_{q \in \mathbb{P}} q^{-2 \sigma} \ll_{\epsilon} 1.
\]
So $|\log F_{t, p, u}(s)| \le C_\epsilon$. Hence

\[
e^{-C_\epsilon} \le F_{t, p, u}(s) \le e^{C_\epsilon}.
\]

\end{proof}

\section{Main results}

Fix a large $N$. Let $X = X(N)$ be a number that we will choose later. Assume that $\log X(N) = o(\sqrt{\log N})$. Let $p$ be an unexceptional prime lying in the interval $[X - X^{4/5}, X]$. By Lemmas  \ref{Siegel} and \ref{short} such $p$ exists.

\begin{theorem} Under assumptions above,
\[
a(r) = \frac{N}{p} + 2\frac{N}{p} \operatorname{Re} \left(e \left(\frac{-r}{p}  A_1 \,  \right) \right) + B_1(r) + B_2(r) + B_3(r).
\]
Here
\[
A_t = \frac{F_{t, p}(1) \, (1 - \frac{p}{p-1}d_{t, p}(1))}{\Gamma(e(t/p))} \log^{e(t/p) - 1} N,
\]
\[
d_{t, p}(s) :=  \left(1 + \frac{e \left( \frac{t}{p} \right)}{p^s} + \frac{e \left( \frac{w(p^2) t}{p} \right)}{p^{2s}} + \ldots\right)^{-1},
\]
\[
F_{t, p}(s) = \prod_{q \in \mathbb{P}} \left(1 + \frac{e \left( \frac{t}{p} \right)}{q^s} + \frac{e \left( \frac{w(q^2) t}{p} \right)}{q^{2s}} + \ldots\right) \left(1 - \frac{1}{q^s} \right)^{e(t/p)},
\]
\[
B_1(r) \ll B_1 := p N \exp\left( -c_6 \sqrt{\log N} \right),
\]
\[
B_2(r) \ll B_2 := N \log^{\cos(2 \pi / p) - 2} N 
\]
and
\[
B_3(r) = 2 \frac{N}{p} \sum_{t=2}^{(p-1)/2} \operatorname{Re} \left( e \left(\frac{-r}{p} \right) A_{t} \right) \ll B_3 := \frac{N}{p^3} \sum_{t=2}^{(p-1)/2} t \log^{\operatorname{cos}(2\pi t / p) - 1} N.
\]
Moreover
\[
\left|A_1 \right| \asymp p^{-2} \log^{\cos(2\pi/p) - 1} N.
\]
\end{theorem}

\begin{proof}

Applying discrete Fourier transform, we obtain
\[
a(r) = \sum_{\substack{n \le N \\ n + w(n) \equiv r \pmod p}} 1 = \frac{1}{p} \sum_{\substack{n \le N \\ t \in \mathbb{Z} / p \mathbb{Z}}} 
e \left( \frac{(n + w(n)) t}{p}\right) e \left( \frac{-r t}{p} \right).
\]
Let $k, t \in \mathbb{Z}/p\mathbb{Z}$. Set
\[
S_t := \sum_{n \le N} e \left( \frac{w(n) t}{p} \right),
\]
\[
S_{k, t} := \sum_{\substack{n \le N \\ n \equiv k \pmod p}} e \left( \frac{w(n) t}{p} \right),
\]
\[
S_{t, \chi} := \sum_{n \le N} e \left( \frac{w(n) t}{p}\right) \chi(n) .
\]

Then
\[
a(r) = \frac{1}{p} \sum_{t, k \in \mathbb{Z}/p \mathbb{Z} } e \left( \frac{(k - r) t}{p} \right) S_{k, t}.
\]

Note that
\[
S_{k, 0} = \frac{N}{p} + O(1),
\]
\[
S_{0, t} = S_t - S_{t, \chi_0},
\]
\[
S_{k, t} = \frac{1}{p-1} \sum_{\chi} \overline{\chi(k)} \, S_{t, \chi}, \, \text{if} \, \, \, k, t \ne 0 .
\]

Putting $u \equiv 1$ in Lemma \ref{F_{t, p, u}}, we obtain
\[
\sum_{n \ge 1} \frac{e \left( \frac{w(n) t}{p} \right)}{n^s} = (\zeta(s))^{e(t/p)} F_{t, p}(s).
\]
Here $F_{t, p}(s) := F_{t, p, u}(s)$. 

Now applying Lemma \ref{Chang}, we see that
\[
S_t = \frac{F_{t, p}(1)}{\Gamma(e(t/p))} N \log^{e(t/p)-1} N + O(N \log^{\operatorname{cos} (2 \pi t / p) - 2}N).
\]

Lemma \ref{F_{t, p, u}} gives us
\[
\sum_{n \ge 1} \frac{e \left( \frac{w(n) t}{p} \right) \chi(n)}{n^s} = (L(s, \chi))^{e(t/p)} F_{t, p, \chi}(s).
\]

If $\chi$ is non-principal, then using Lemma \ref{L_main} we see that
\[
S_{t, \chi} = O\left(N \exp\left( -c_6 \sqrt{\log N} \right)\right).
\]

Note that $c_6$ and constant in $O$ do not depend on $\chi$.

Therefore, if $k, t \ne 0$, then we have
\[
S_{k, t} = \frac{S_{t, \chi_0}}{p-1} + O\left(N \exp\left( -c_6 \sqrt{\log N} \right)\right),
\]

there $\chi_0$ is the principal character modulo $p$. Let us denote
\[
d_{t, p}(s) :=  \left(1 + \frac{e \left( \frac{t}{p} \right)}{p^s} + \frac{e \left( \frac{w(p^2) t}{p} \right)}{p^{2s}} + \ldots\right)^{-1}.
\]
 We have

\[
\sum_{n \ge 1} \frac{e\left(\frac{w(n) t}{p}\right) \chi_0(n)}{n^s} = 
d_{t, p}(s) \sum_{n \ge 1} \frac{e \left( \frac{w(n) t}{p} \right)}{n^s} =
(\zeta(s))^{e(t/p)} F_{t, p}(s) d_{t, p}(s) .
\]

Using Lemma \ref{Chang}, we obtain
\[
S_{t, \chi_0} = \frac{F_{t, p}(1) d_{t, p}(1)}{\Gamma(e(t/p))} N \log^{e(t/p)-1} N + O(N \log^{\operatorname{cos} (2 \pi t / p) - 2}N),
\]

Putting all together we obtain
\[
a(r) = \frac{N}{p} + \frac{1}{p} \sum_{t \not\equiv 0 \pmod p} e (-rt/p) S_{0, t}  + 
\frac{1}{p (p-1)} \sum_{t \not\equiv 0 \pmod p} S_{t, \chi_0} \sum_{k \not\equiv 0 \pmod p} e \left( \frac{(k - r) t}{p} \right) + B_1(r).
\]
Here
\[
B_1(r) \ll p N \exp\left( -c_6 \sqrt{\log N} \right).
\]

Thus
\[
a(r) = \frac{N}{p} + \frac{1}{p} \sum_{t \not\equiv 0 \pmod p} e(-rt/p) (S_t - S_{t, \chi_0} - S_{t, \chi_0}/(p-1)) + B_1(r) = 
\]
\[
\frac{N}{p} + \frac{1}{p} \sum_{t \not\equiv 0 \pmod p}  \left(S_t - \frac{p}{p-1} S_{t, \chi_0} \right) e(-rt/p) + B_1(r).
\]

Now we have
\[
a(r) = \frac{N}{p} + \frac{N}{p} \sum_{t \not\equiv 0 \pmod p} \frac{e(-rt/p) F_{t, p}(1) \, (1 - \frac{p}{p-1}d_{t, p}(1))}{\Gamma(e(t/p))} \log^{e(t/p) - 1} N + B_1(r) + B_2(r).
\]
Here
\[
B_2(r) = O\left( N \log^{\cos(2 \pi / p) - 2} N \right).
\]

Let us denote
\[
A_t := \frac{F_{t, p}(1) \, (1 - \frac{p}{p-1}d_{t, p}(1))}{\Gamma(e(t/p))} \log^{e(t/p) - 1} N,
\]
\[
B_3(r) = \frac{N}{p} \sum_{t=2}^{p-2} e(-rt/p) A_t.
\]

Then
\[
a(r) = \frac{N}{p} + 2\frac{N}{p} \operatorname{Re} \left( e \left(\frac{-r}{p} \right) A_{1} \right) + B_1(r) + B_2(r) + B_3(r).
\]

We have
\[
B_3(r) = 2 \frac{N}{p} \sum_{t = 2}^{(p-1)/2} \operatorname{Re} \left( e \left(\frac{-r t}{p} \right) A_{t} \right).
\]

Note that
\[
\left|1 - \frac{p}{p-1}d_{t,p}(1) \right| \asymp \left|(1 - p^{-1}) - \left( 1 - \frac{e(t/p)}{p} + \frac{e(t/p)^2 - e(w(p^2) t / p)}{p^2} + O(p^{-3}) \right) \right| =
\]
\[
\left|(1 - p^{-1}) - \left(1 - \left(1 + \frac{2 \pi i t}{p} + O(t^2 / p^2)\right) p^{-1} + \theta p^{-2}\right)\right| = 2 \pi i \frac{t}{p^2} + 2 \theta p^{-2} + O\left( \frac{t^2}{p^3} \right) \ll \frac{|t|}{p^2},
\]
here $|\theta| \le 1$. If $t = 1$, we obtain
\[
\left|1 - \frac{p}{p-1}d_{t,p}(1) \right|  \gg p^{-2}.
\]

Since $\Gamma(s)^{-1}$ is an entire function and $|F_{t, p}(1)| \le b_2$, it follows that
\[
\left| A_{t} \right| \ll \frac{t}{p^2} \log^{\cos(2\pi t/ p) - 1} N.
\]

Hence
\[
B_3 \ll \frac{N}{p^3} \sum_{t=2}^{(p-1)/2} t \log^{\operatorname{cos}(2\pi t / p) - 1} N.
\]

Since
$|e(1/p) - 1| \ll p^{-1}$ and $\Gamma(1) = 1$, it follows that 
\[
\left|\Gamma(e(1/p))^{-1} \right| > C_{\Gamma}
\]
for some constant $C_{\Gamma} > 0$ and $p$ large enough. Further we have $\left| F_{t, p} (1) \right| \ge b_1$. Hence
\[
\left|A_1 \right| \gg p^{-2} \exp((\cos(2\pi/p) - 1) \log \log N).
\]

This concludes the proof.

\end{proof}

\begin{proof}[Proof of Theorem \ref{main}]

Let $r$ belong to $R \subset \mathbb{Z}/p \mathbb{Z}$ iff
\[
2\frac{N}{p} \operatorname{Re} \left( e \left(\frac{-r}{p} \right) A_{1} \right) < 0.
\]
Note that 
\[
\left| \left\{ r : \operatorname{Re} \left( e \left(\frac{-r}{p} \right) A_{1} \right) \le -\frac{1}{2} \left| A_{1} \right|\right\} \right| \asymp p.
\]

Thus, substituting such $R$ into Lemma \ref{idea}, we obtain
\[
\Xi(N) \gg \sum_{r \in R} \frac{N}{p} - a(r) - 1 \gg -\sum_{r \in R} \left(  2\frac{N}{p} \operatorname{Re} \left( e \left(\frac{-r}{p} \right) A_{1} \right) + O(B_1 + B_2 + B_3 + 1) \right) \gg
\]
\[
N \left|A_{1}\right| + p O(B_1 + B_2 + B_3 + 1).
\]

Note that
\[
\frac{X^2}{p^2} \le 1 + 5 X^{-1/5}.
\]

Let us choose $X = \alpha^{-1} \sqrt{\log \log N}$, we obtain
\[
N \left|A_1\right| \gg N p^{-2} \exp \left( ((\cos(2\pi/p)) - 1) \log \log N \right) \gg N p^{-2} \exp\left( - \frac{2\pi^2}{p^2} \log \log N \right) \gg
\]
\[
 N X^{-2} \exp\left( - \frac{2\pi^2}{X^2}  (1 + 5 X^{-1/5}) \log \log N \right) \gg
\]
\[
 N X^{-2} \exp\left( - \frac{2\pi^2}{X^2} \log \log N \right) \gg \alpha^2 \exp \left(- 2\pi^2 \alpha^2 \right) \frac{N}{\log \log N}.
\]

Let us choose $\alpha$ such that $p B_3 \le 0.1 \, N \left|A_1\right|$. That is
\[
\tilde{C} \, \sum_{t=2}^{(p-1)/2} t \log^{\operatorname{cos}(2\pi t / p) - 1} N \le   \exp \left(- 2 \pi^2 \alpha^2 \right)
\]
for some constant $\tilde{C}$.

We have
\[
\sum_{t=2}^{(p-1)/2} t \log^{\operatorname{cos}(2\pi t / p) - 1} N  = 
\sum_{2 \le t < p/4} t \log^{\operatorname{cos}(2\pi t / p) - 1} N  + O(p^2 \log^{-1} N).
\]

Note that
\[
\frac{5 \pi^2}{X^2} \le \cos \frac{2\pi}{p} - \cos \frac{4\pi}{p} \le \cos \frac{4\pi}{p} - \cos \frac{6\pi}{p} \le \ldots 
\]

Therefore,
\[ 
\sum_{2 \le t < p/4} t \log^{\operatorname{cos}(2\pi t / p) - 1} N  \le 
\sum_{2 \le t < p/4} t \log^{\frac{-5 \pi^2 (t - 1)}{X^2}} N \le
\]
\[
2 \sum_{t = 1}^{\infty} t \exp \left(-5 \pi^2 \alpha^2 t \right) = \frac{2 \exp(-5 \pi^2 \alpha^2)}{\left( 1 - \exp(-5 \pi^2 \alpha^2) \right)^2}.
\]

Clearly, for $\alpha$ large enough, we have
\[
\frac{2 \exp(-5 \pi^2 \alpha^2)}{\left( 1 - \exp(-5 \pi^2 \alpha^2) \right)^2} \le 0.5 \, \tilde{C}^{-1} \exp \left(- 2\pi^2 \alpha^2\right).
\]
Obviously,
\[
p^2 \log^{-1} N = o(1).
\]
Thus we obtain 
\[
p B_3 \le 0.1 N \left| A_{1} \right|
\]
for $\alpha$ large enough. Note that $\alpha$ does not depend on $N$. Hence
\[
N \left|A_{1}\right| \gg \frac{N}{\log \log N}. 
\]

Obviously, $pB_1 = o \left( N / \log \log N \right)$ and $pB_2 = o \left( N / \log \log N \right)$.

Therefore, 
\[
\Xi(N) \gg N \left|A_{1}\right| \gg \frac{N}{\log \log N}.
\]

\end{proof}

I am grateful to Vitalii V. Iudelevich for setting the problem and Alexander B. Kalmynin for valuable discussions.




\begin{thebibliography}{99}
\bibitem{Karatsuba}
A. A. Karatsuba, Melvyn B. Nathanson. Basic Analytic Number Theory, Springer-Verlag Berlin Heidelberg, 1993
\bibitem{Changa_book}
M. E. Changa. Methods of analytic number theory, RCD, Moscow–Izhevsk 2013.
(Russian)
\bibitem{Changa}
M. E. Changa. On integers whose number of prime divisors belongs to a given residue class 2019 Izv. Math. 83-173
\bibitem{Erdos1}
 P. Erd\H{o}s, A. S\'ark\H{o}zy, C. Pomerance. On locally repeated values of certain arithmetic function, I, 1983
\bibitem{Cons_primes}
R. C. Baker, G. Harman, J. Pintz. The Difference Between Consecutive Primes, II,
Proceedings of the London Mathematical Society, Volume 83, Issue 3, November 2001, Pages 532–562
\end{thebibliography}
\end{document}